\documentclass[10pt]{amsart}
\usepackage{amsfonts}
\usepackage{amsthm,amsmath,amssymb,mathrsfs,yhmath}
\usepackage{tikz}
\usepackage{pstricks,pst-node, pst-text,pst-3d,graphicx}
\usepackage{xcolor}
\usepackage{anysize}
\usepackage{enumerate}
\usepackage{verbatim}
\usepackage[active]{srcltx}
\usepackage{refstyle}

\marginsize{2.5cm}{3cm}{3cm}{4cm}
\usetikzlibrary{trees}
\newtheorem{theorem}{Theorem}
\newtheorem{corollary}[theorem]{Corollary}

\newtheorem{proposition}[theorem]{Proposition}
\newtheorem{example}[theorem]{Example}

\theoremstyle{definition}

\newtheorem{definition}{Definition}

\theoremstyle{remark}
\newtheorem{remark}{Remark}

\newcommand{\rr}{\mathbb{R}}

\newcommand{\dbs}{\overline{\dim}_B\hspace{.05cm}}
\newcommand{\dbi}{\underline{\dim}_B\hspace{.05cm} }

\newcommand{\diam}{\textrm{diam}\: }
\newcommand{\dist}{\textrm{dist}}

\begin{document}
\title[Quasi-Assouad Dimension]{Properties of Quasi-Assouad Dimension}

\author{Ignacio Garc\'{\i}a}
\address{Centro Marplatense de Investigaciones Matem\'aticas (CIC),
Facultad de Ciencias Exactas y Naturales\\
Instituto de Investigaciones F\'{\i}sicas de Mar del Plata (CONICET)\\
Universidad Nacional de Mar del Plata, Argentina}
\email{nacholma@gmail.com}

\author{Kathryn Hare}
\address{Dept. of Pure Mathematics\\
University of Waterloo\\
Waterloo, Ont., Canada N2L 3G1}
\email{kehare@uwaterloo.ca}
\subjclass[2010]{Primary 28A80, 28A78}
\keywords{Assouad dimension, weak tangents, orthogonal projections}

\maketitle

\begin{abstract}
The connections between quasi-Assouad dimension and tangents are
studied. We apply these results to the calculation of the quasi-Assouad dimension for a class of planar self-affine sets. We also show that sets with decreasing gaps have quasi-Assouad dimension $0$ or $1$ and exhibit an example of a set in the plane whose quasi-Assouad dimension is smaller than that of its projection onto the $x$-axis, showing that quasi-Assouad dimension may increase under Lipschitz
mappings. Moreover, for closed sets, we show that the Hausdorff dimension is an upper bound for the lower-Assouad dimension.
\end{abstract}

\section{Introduction}

Recently, a number of authors have investigated the Assouad and
lower-Assouad dimensions of subsets of $\mathbb{R}^{d}$; see \cite{Fr} and
the many papers referenced there. These dimensions differ from the
well-known Hausdorff and box dimensions as they provide information about
the extreme behaviour of the local geometry of the set. The quasi-Assouad
dimensions give less extreme, but still local, geometric information. As in
the case of Assouad dimensions, quasi-Assouad dimensions take into account
relative scales locally at any position in the set, but the difference is
that the `minimal depth' of the relative scales increases as the size of the
considered neighbour decreases, and this has a moderating effect. These
dimensions were introduced by L\"{u} and Xi in \cite{LX} and shown to be an
invariant quantity under quasi-Lipschitz maps, unlike the Assouad dimension. In this note, we develop
further properties of the quasi-Assouad dimensions.

To give their definitions we first introduce notation. For a subset $%
A\subseteq \mathbb{R}^{d}$, let $N_{r}(A)$ be the least number of balls of
radius $r$ needed to cover $A$ and let $N_{r,R}(E)=\max_{x\in E}N_{r}(E\cap
B(x,R))$. The upper and lower box dimensions, denoted $\overline{\dim }_{B}%
\hspace{0.05cm}E$ and $\underline{\dim }_{B}\hspace{0.05cm}E$ respectively,
are given by 
\begin{equation*}
\overline{\dim }_{B}E=\limsup_{r\rightarrow 0}\frac{\log N_{r}(E)}{\log r},%
\text{ }\underline{\dim }_{B}\hspace{0.05cm}E=\liminf_{r\rightarrow 0}\frac{%
\log N_{r}(E)}{\log r}
\end{equation*}%
If the two values coincide, we refer to it as the box dimension, $\dim _{B}E$%
. Given $0\leq \delta <1$, define 
\begin{equation*}
\overline{h_{E}}(\delta )=\inf \bigg\{\alpha \geq 0:\exists b,c>0\text{ such
that }\forall \text{ }0<r\leq R^{1+\delta }\leq R\leq b,N_{r,R}(E)\leq
c\left( \frac{R}{r}\right) ^{\alpha }\bigg\}
\end{equation*}%
and 
\begin{equation*}
\underline{h_{E}}(\delta )=\sup \bigg\{\alpha \geq 0:\exists b,c>0\text{
such that }\forall \text{ }0<r\leq R^{1+\delta }\leq R\leq b,N_{r,R}(E)\geq
c\left( \frac{R}{r}\right) ^{\alpha }\bigg\}.
\end{equation*}%
These functions are monotone and by taking limits we obtain the
quasi-Assouad dimension of $E$, $\dim _{qA}E$, and the quasi-lower Assouad
dimension of $E$, $\dim _{qL}E$: 
\begin{equation*}
\dim _{qA}E=\lim_{\delta \rightarrow 0}\overline{h_{E}}(\delta )\text{, }%
\dim _{qL}E=\lim_{\delta \rightarrow 0}\underline{h_{E}}(\delta ).
\end{equation*}%
The Assouad and the lower-Assouad dimensions are given by 
\begin{equation*}
\dim _{A}E=\overline{h_{E}}(0)\text{, }\dim _{L}E=\underline{h_{E}}(0).
\end{equation*}

For any bounded set $E$ these dimensions are ordered in the following way: 
\begin{equation}
\dim _{L}E\leq \dim _{qL}E\leq \underline{\dim }_{B}\hspace{0.05cm}E\leq 
\overline{\dim }_{B}\hspace{0.05cm}E\leq \dim _{qA}E\leq \dim _{A}E.
\label{ordering}
\end{equation}%
It is known that all these dimensions coincide for self-similar sets
satisfying the open set condition, but for more general sets, strict
inequalities are possible throughout, c.f. \cite{Fr, LX} and Example \ref%
{strict} in Section \ref{section:further_results}. Moreover, in that section we show that inequality $\dim _{qL}E\leq \dim _{H}E$ holds for closed sets. This inequality is not true in general, since $\dim_H\mathbb{Q}=0$ but $\dim_{qL}\mathbb{Q}=1$.

Generalizations of tangents of a set are a\ useful concept when considering the
Assouad dimension. These are essentially limits, in the Hausdorff metric, of
sequences of magnifications of local parts of the set. Tangents often have
simpler structure than the original set and their Assouad dimensions are
always lower bounds for the Assouad dimension of the original set (\cite{Fr,
Ma, MT}). In Section \ref{section:fast_tangent}, we show that if the convergence to the tangent is
sufficiently fast, and the lower-Assouad and Assouad dimensions of the
tangent coincide, then this value is a lower (upper) bound for the
quasi-(lower) Assouad dimension of the original set. From this, we show that
the quasi-(lower) Assouad dimensions of a class of self-affine carpets are the same as
their Assouad dimensions.

In Section \ref{section:further_results} we observe that if $\overline{\dim }_{B}\hspace{0.05cm}E=0$,
then the same is true for the quasi-Assouad (but not necessarily, the
Assouad) dimension of $E$ and show that the quasi-Assouad dimension of a
sequence in $\mathbb{R}$ with decreasing gaps is either $0$ or $1$. This
same dichotomy was shown to hold for Assouad dimensions (although not
necessarily with the same value for the Assouad and quasi-Assaoud
dimensions) in \cite{GHM}. We also give an example to illustrate that, like
the Assouad dimension (but not the Hausdorff or box dimensions), the
quasi-Assouad dimension can rise when taking projections.

\section{Quasi-Assouad dimension and Tangent structure}\label{section:fast_tangent}

In this section we find bounds on the quasi-Assouad dimension of a set by
dimensions of its `tangents'. We begin by recalling some definitions and
results. Given $X,Y\subset \mathbb{R}^{d}$ compact subsets, their Hausdorff
distance is 
\begin{equation*}
\dist_{H}(X,Y)=\max \{p_{H}(X,Y),p_{H}(Y,X)\},
\end{equation*}%
where 
\begin{equation*}
p_{H}(X,Y)=\sup_{x\in X}\inf_{y\in Y}\Vert x-y\Vert .
\end{equation*}

\begin{definition}
Let $F$ and $\hat{F}$ be compact subsets of $\mathbb{R}^{d}$. We say that $%
\hat{F}$ is a \emph{weak tangent} of $F$ if there is a compact subset $%
X\subset \mathbb{R}^{d}$, that contains both $F$ and $\hat{F}$, and a
sequence of bi-Lipschitz maps $T_{k}:\mathbb{R}^{d}\rightarrow \mathbb{R}%
^{d} $, with Lipschitz constants $a_{k},b_{k}$ satisfying 
\begin{equation*}
a_{k}\Vert x-y\Vert \leq \Vert T_{k}(x)-T_{k}(y)\Vert \leq b_{k}\Vert
x-y\Vert
\end{equation*}%
and $\sup b_{k}/a_{k}<\infty $, such that $\dist_{H}(T_{k}(F)\cap X,\hat{F}%
)\rightarrow 0$ as $k\rightarrow \infty $. In the case that the contraction
ratios, $b_{k}$, are unbounded we call $\hat{F}$ a \emph{generalized tangent} of $F$.
\end{definition}

The usefulness of tangents is that they can be used to obtain lower bounds
for the Assouad dimension of the original set, namely

\begin{theorem}
\label{thmtangent} If $\hat{F}$ is a weak tangent of $F,$ then 
\begin{equation}
\dim _{A}\hat{F}\leq \dim _{A}F.  \label{boundtangent}
\end{equation}%
If, in addition, there is some $\theta >0$ such that for all $r>0$ and $x\in 
\hat{F}$ there is some $y\in \hat{F}$ such that $B(y,r\theta )\subseteq
B(x,r)\cap X$, then $\dim _{L}F\leq \dim _{L}\hat{F}$.
\end{theorem}

The above result is due to Fraser \cite[Prop. 7.7]{Fr}. A version using
similarities to define tangents (the usual definition) was previously
obtained by Mackay and Tyson, see \cite[Prop. 2.9]{Ma} or \cite[Prop. 6.1.5]%
{MT}.

There is no loss of generality in assuming that $X=[0,1]^{d}$ since the Assouad dimension of any set can be
characterized by its tangents in this sense. To be specific 
\begin{equation}
\dim _{A}F=\max \bigl\{\dim_H \hat{F}:\hat{F}\ \text{ is a generalized tangent of }F\text{
with }X=[0,1]^{d}\bigr\}.  \label{furstenberg}
\end{equation}%
We refer the reader to \cite{CWW}, \cite{Fu1}, \cite{Fu2} and \cite{KOR}, noting that
the terminology there is quite different: microsets and star dimension are
used instead of tangents and Assouad dimension. Observe that from (\ref%
{boundtangent}) and (\ref{furstenberg}) any compact set $F$ has a generalized tangent $%
\hat{F}$ for which 
$
\dim _{H}\hat{F}=\dim _{A}\hat{F}=\dim _{A}F.
$


We call the set $\hat{F}$ a \emph{pseudo-tangent} of $F$ if there is a
sequence of bi-Lipschitz maps $T_{k},$ as above, with $\sup b_{k}=\infty $,   such that we have the
`one-sided' pseudo-distance, $p_{H}(\hat{F},T_{k}(F))\rightarrow 0$ as $%
k\rightarrow \infty $. The convenience of this definition is that it
does not require the intersection with the set $X$ or the two-sided
comparison of distance, and still the inequality $\dim _{A}\hat{F}\leq \dim
_{A}F$ holds in this case; see \cite{FHOR}, where the definition and proof
are made for similarities, but the same proof applies in the bi-Lipschitz
setting.

Returning to the quasi-Assouad dimension, we first give an example where (%
\ref{boundtangent}) fails when Assouad is replaced by quasi-Assouad
dimension. The following result is useful for this task; a more general
version in $\mathbb{R}^{d}$ was obtained recently (\cite{FY3}), but here we
include a different proof.

\begin{proposition}
A subset $F\subseteq \mathbb{R}$ has Assouad dimension $1$ if and only if $%
[0,1]$ is a generalized tangent of $F$.
\end{proposition}

\begin{proof}
By (\ref{boundtangent}), the Assouad dimension of any tangent is a lower
bound for the Assouad dimension of the set. On the other hand, by \cite[Thm.
5.1.8]{MT}, the assumption that $\dim _{A}F=1$ is equivalent to the fact
that $F$ is not uniformly disconnected. Thus for each $k$ there are distinct
points, $z_{0}^{k},\ldots ,z_{n}^{k}$ $\in $ $F$, such that $%
|z_{i}-z_{i+1}|<|z_{0}-z_{n}|/k$, where $z_{0}<z_{i}<z_{n}$ and $%
|z_{0}-z_{n}|$ goes to $0$. The last condition ensures that $[0,1]$ is a generalized
tangent (and not just a weak tangent).

Let $T_{k}$ be the affine transformation that maps $[z_{0},z_{n}]$ to $[0,1]$. For $x_{i}=T_{k}(z_{i})$ we have 
\begin{equation*}
|x_{i}-x_{i+1}|=\frac{|z_{i}^{k}-z_{i+1}^{k}|}{|z_{n}^{k}-z_{0}^{k}|}<\frac{1%
}{k}
\end{equation*}%
and hence for $0\leq j<k$ each interval $\bigl[j/k,(j+1)/k\bigr)$ contains
at least one of the $x_{i}$'s. This shows the sets $T_{k}(\{z_{0}^{k},\ldots
,z_{n}^{k}\})$ converge in the Hausdorff metric to $[0,1]$ and therefore, $%
\dist_{H}(T_{k}(F)\cap \lbrack 0,1],[0,1])\rightarrow 0$ as $k\rightarrow
\infty $.
\end{proof}

\begin{corollary}
If $F\subseteq \mathbb{R}$ has $\dim _{qA}F<1=\dim _{A}F,$ then $\hat{F}%
=[0,1]$ is a generalized tangent with 
\begin{equation*}
1=\dim _{H}\hat{F}=\dim _{qA}\hat{F}>\dim _{qA}F.
\end{equation*}
\end{corollary}

Hence, not only (\ref{boundtangent}), but also (\ref{furstenberg}), fails
when Assouad dimension is replaced by quasi-Assouad dimension. An explicit
example can be given by considering the Cantor-like sets $F_{\alpha }$
constructed in the following way.

\begin{example}
Fix $0<\alpha <1$ and for each $k\geq 1$ put $D^{k}=\{(i_{1},\ldots
,i_{k}):1\leq i_{j}\leq 2^{j},1\leq j\leq k\}$. We construct the set
inductively, beginning with $I_{\phi }=[0,1]$ at step $0$. Having
constructed the step $k-1$ Cantor intervals, $I_{i_{1}\ldots i_{k-1}}$, for $%
(i_{1},\ldots ,i_{k-1})\in D^{k-1},$ to construct the Cantor intervals of
step $k$ we take the $2^{k}$ subintervals, $I_{i_{1}\ldots i_{k-1}l}$, $%
1\leq l\leq 2^{k},$ uniformly distributed inside $I_{i_{1}\ldots i_{k-1}},$
with equal lengths $2^{-k/\alpha }|I_{i_{1}\ldots i_{k-1}}|$. The
Cantor-like set $F_{\alpha }$ is given by 
\begin{equation*}
F_{\alpha }=\bigcap_{k\geq 1}\bigcup_{\mathbf{i}\in D^{k}}I_{\mathbf{i}}.
\end{equation*}%
Example 1.17 from \cite{LX} is a special case and it follows from the
arguments given there that $\dim _{qA}F_{\alpha }=\alpha $, while $\dim
_{A}F_{\alpha }=1$.
\end{example}

Moreover, this example turns out to be quite pathological with respect to
tangents.

\begin{proposition}
\label{proposition-quasi-no-tg} Any generalized tangent of the set $F_{\alpha }$ of the
example above is either a finite set or an interval.
\end{proposition}

\begin{proof}
For notational ease, we will omit the subscript $\alpha $. Let $\hat{F}$ be
a generalized tangent of $F$ and denote by $T_{k}$ the associated bi-Lipschitz maps. By
a gap of $\hat{F}$ we mean a bounded complementary (maximal) open interval
of the complement of $\hat{F}$. If $\hat{F}$ has no gaps, then it is an
interval (which may be a singleton). So, assume that $\hat{F}$ contains at
least one gap. Choose one of maximal length, say $J=(a,b)$, where $a,b\in 
\hat{F}$. Below we show that $i)$ if $x\in \hat{F}$ and $x<a$, then $a-x\geq
|J|$, or $ii)$ if $x\in \hat{F}$ and $x>b$, then $x-b\geq |J|$. This
statements imply, by the maximality of $J$, that $\hat{F}$ is a finite
equidistributed set in some subinterval of $[0,1]$.

We show only $i)$ since $ii)$ follows by a symmetric argument. So assume
that $x<a$ with $x\in \hat{F}$. One consequence of the Hausdorff convergence
is that for each sufficiently small $\epsilon >0$ (much smaller than $|J|$
and $a-x$), and all sufficiently large $k$, the sets $(a-\epsilon
,a+\epsilon )$ and $(b-\epsilon ,b+\epsilon )$ contain points of $T_{k}(F)$,
while $(a+\epsilon ,b-\epsilon )\cap T_{k}(F)=\emptyset $. Hence there is a gap $%
G_{k}$ of $F$ from some step ${k}$ in the construction of $F$ such that 
\begin{equation*}
(a+\epsilon ,b-\epsilon )\subseteq T_{k}(G_{k})\subseteq (a-\epsilon
,b+\epsilon ).
\end{equation*}%
Suppose $I_{k}^{L}$ is the closed interval from step ${k}$ in the
construction of $F$ that shares an endpoint with $G_{k}$ and is placed on
its left. The set $F$ has the property that $|I_{k}^{L}|/|G_{k}|\rightarrow
0 $ and as the maps $T_{k}$ are bi-Lipschitz with constants satisfying $\sup
b_k/a_k<\infty$, this ensures that $T_{k}(I_{k}^{L})\subseteq (a-2\epsilon
,a+2\epsilon )$ for large enough $k$. Note that the gap $G_{k}^{L},$
adjacent to $I_{k}^{L}$ but on its left, and whose existence is guaranteed because of the existence of $x$ and the choice of $\epsilon,$ is a gap from some step $\tilde k<k$. Note also that the lengths $g_{k-1}$ and $g_k$ of any gaps from steps $k-1$ and $k$ verify $g_k/g_{k-1}\to 0$ as $k\to\infty$.  In particular, for every $k$ sufficiently large we get $b_k|G_k|\le a_k|G_k^L|$.

Thus 
\begin{equation*}
|J|-2\epsilon \leq |T_{k}(G_{k})|\leq |T_{k}(G_{k}^{L})|,
\end{equation*}%
so any point in $\hat{F}$ to the left of $a$, and not contained in $%
(a-2\epsilon ,a+2\epsilon )$, must be at least distance $|J|-2\epsilon$
from $a$ for some constant $c>0$. But since $\epsilon $ can be made
arbitrarily small, an easy argument shows that $\hat{F}\cap (a-2\epsilon
,a+2\epsilon )=\{a\}$, and therefore we conclude that $i)$ holds by letting $%
\epsilon \rightarrow 0$.
\end{proof}

Note that the above example also illustrates that there is no way to select
a subfamily from the tangents to $F$ so that the identity (\ref{furstenberg}%
) remains valid for the quasi-Assouad dimension. However, we can extend
Theorem \ref{thmtangent} above to the quasi-Assouad dimensions if we
restrict to tangents for which the convergence is sufficiently
quick. 

\begin{definition}
We say that the generalized tangent $\hat{F}$ to the set $F$ is a \emph{generalized fast
tangent} if the following decay condition is satisfied: there are constants $%
C,\epsilon >0$ such that $${\rm dist}_{H}(T_{k}(F)\cap \lbrack 0,1]^{d},\hat{%
F})\leq Cb_{k}^{-\epsilon },$$ where $T_k$ and $b_k$ are as in the definition of a generalized tangent. In this case we say that $\hat F$ is a \emph{tangent of order $\epsilon$}. We similarly define \emph{fast pseudo-tangents 
}by the requirement that the pseudo-tangent $\hat{F}$ verifies $p_{H}(\hat{F}%
,T_{k}(F))\leq Cb_{k}^{-\epsilon }$. 
\end{definition}

 The relation between tangents and quasi-Assouad dimensions is given in the next result, where we have chosen to weaken some hypotheses for clarity of the exposition; see Remark \ref{remtheotan} for more general statements.

\begin{theorem}\label{theotan} Suppose $\hat{F}\subset\rr^d$ is a non-empty, generalized fast tangent of $F\subset\rr^d$ given by  bi-Lipschitz maps $T_{k}$ with Lipschitz constants $a_{k},b_{k}$ satisfying $\sup
b_{k}=\infty $ and $\sup b_{k}/a_{k}<\infty$. 

(i) 
For the quasi-Assouad dimension of $F$ we have the lower bound $$\underline{\dim }_{B}\hat{F}\leq \dim _{qA}F.$$ If, in addition, there is some $C'$ such that $
b_{k+1}\leq C'b_{k}$, then $$\dim_{qA}\hat{F}\leq \dim_{qA}F.$$

(ii) Assume that $\hat F$ contains an interior point of $[0,1]^d$. Then, for the quasi-lower Assouad dimension we have the upper bound $$\dim _{qL}F \leq \overline{%
\dim }_{B}\hspace{0.05cm}\hat{F}.$$ 
If, in addition, there is some $C'$ such that $
b_{k+1}\leq C'b_{k}$ and furthermore, there is $\theta>0$ such that for any $r\in(0,1]$ and $x\in\hat F$ there is $y\in \hat F$ such that $B(y, r\theta)\subset B(x,r)\cap[0,1]^d$, then $$\dim _{qL}F \leq \dim_{qL}\hat{F}.$$
\end{theorem}

\begin{remark}
In order to better understand how the inhomogeneity of a set depends on the
scale, Fraser and Yu \cite{FY1} considered the 
refined parametric variants of the quasi-Assouad and quasi-lower Assouad dimensions, $\dim _{A}^{\theta }F$ and 
$\dim _{L}^{\theta }F$, for $\theta \in (0,1)$, known as the Assouad
spectrum and lower spectrum of $E$ respectively; see \cite{FY1} for the precise definitions. Similar statements as in Theorem \ref{theotan} can be made for $\dim _{A}^{\theta }F$ and 
$\dim _{L}^{\theta }F$, where the allowable $\theta $ depend on the choice of 
$\varepsilon $. For example, if $\hat F$ is a generalized fast tangent of order $\epsilon$ such that $\dbi \hat F=s$, then $\dim_A^{\theta}F\ge s$ for any $1/(1+\epsilon)\le \theta<1$. We leave the technical details for the reader. 
\end{remark}

Notation: When we write $x_k\approx X_k$ we mean there are positive constants $a,b$ such that $aX_k \le x_k \le bX_k$ for all $k$.

\begin{proof}
(i) Let \underline{$\dim $}$_{B}\hat{F}=s$. We may assume $s>0$, else the
result is trivial. Temporarily fix $\eta >0$. Then
\begin{equation*}
N_{r}(\hat{F})\geq r^{-(s-\eta )}
\end{equation*}%
for all sufficiently small $r$.

Since $\hat F$ is a generalized fast tangent, there are constants $C, \epsilon>0$ such that $\dist_H(T_k(F)\cap[0,1]^d, \hat F)\le Cb_k^{-\epsilon}$. Put $r_{k}=Cb_{k}^{-\varepsilon }$, 
let $R=\diam\hat{F}$ and pick any $y_{0}\in \hat{F}$. Then, for each $k$ we can find $m=m_{k}\approx \left( r_{k}^{-(s-\eta )}\right) $ points, $y_{1},\ldots ,y_{m}\in
B(y_{0},R)$ that are $3r_{k}$-separated. The assumption on $\hat F$ implies, in particular, that  $p_{H}(\hat{F}, T_k(F)%
)\leq r _{k}$, and this ensures that we can choose $x_{0},x_{1},\dots ,x_{m}\in T_{k}(F)$
such that $\Vert x_{i}-y_{i}\Vert \leq r_{k}$ for each $i=0,\ldots ,m$. We
have $x_{i}\in B(x_{0},R+2r_{k})$ and $\Vert x_{i}-x_{j}\Vert \geq r_{k}$
for all $1\leq i$ $\neq j\leq m$.

Taking preimages under $T_{k},$ we can find $z_{0},z_{1},\ldots ,z_{m}$ $\in
F$ such that 
\begin{equation*}
z_{i}\in B\bigl(z_{0},\frac{1}{a_{k}}(R+2r_{k})\bigr)\text{ \ \ and \ \ }%
\left\Vert z_{i}-z_{j}\right\Vert \geq \frac{r_{k}}{b_{k}}\text{.}
\end{equation*}%
This shows that 
\begin{equation*}
N_{\frac{r_{k}}{b_{k}}}\bigl(F\cap B(z_{0},\frac{1}{a_{k}}(R+2r_{k}))\bigr)%
\geq m.
\end{equation*}%
Note that $m\approx \left( \frac{(R+2r_{k})/a_{k}}{r_{k}/b_{k}}\right) ^{s-\eta
}$ (and hence $\dim _{A}F\geq s)$. An easy calculation shows 
\begin{equation*}
\frac{r_{k}}{b_{k}}\leq \left( \frac{1}{a_{k}}(R+2r_{k})\right) ^{1+\varepsilon /2}
\end{equation*} for large $k$, and, of course, $(R+2r_{k})/a_{k}\rightarrow 0,$ consequently $\dim _{qA}F\geq s-\eta $. As $%
\eta $ $>0$ is arbitrary, $\dim _{qA}F\geq s$.

Now suppose that $\dim _{qA}\hat{F}=t$ and there is some $C^{\prime }$ such
that $b_{k+1}\leq C^{\prime }b_{k}$ for all $k$. Again, temporarily fix $%
\eta >0$. Then there is some $0<\delta <1$ and arbitrarily small $r,R$ with $%
r\leq R^{1+\delta }$ and $y_{0}\in \hat{F}$ such that $N_{r}(B(y_{0},R)\cap 
\hat{F})\geq (R/r)^{t-\eta }$. Choose $k$ such that $b_{k+1}^{-\varepsilon
}\leq R\leq b_{k}^{-\varepsilon }$. As above, we deduce that for suitable $%
z_{0}\in F$ we have
\begin{equation}\label{bound}
N_{\frac{r}{b_{k}}}\bigl(F\cap B(z_{0},\frac{1}{a_{k}}(R+2r))\bigr)\geq
\left( \frac{R}{r}\right) ^{t-\eta }.
\end{equation}%
Since $R\approx b_{k}^{-\varepsilon }$ and $r\leq R^{1+\delta }$, one can easily
verify that $r/b_{k}\geq $ $\left( (R+2r)/a_{k}\right) ^{1+\sigma }$ for a
choice of $\sigma >0$ (depending on $\delta $ and $\varepsilon $). It
follows that  $\dim _{qA}F\geq t$.

(ii) Let $\dbs \hat{F}=s$.  Given any $\eta >0$ we have $N_{r_k}(\hat F)\leq r_{k}^{-(s+\eta)}$ if $k$ is sufficiently large, where as before $r_k=Cb_k^{-\epsilon}$ with $C$ and $\epsilon$ given by the definition of the generalized fast tangent. The generalized fast tangent hypothesis implies that for each $x\in T_k(F)\cap[0,1]$ there is some $\hat x\in\hat F$ such that $\|x-\hat x\|\le r_k$ and this ensures that $N_{3r_k}(T_k(F)\cap[0,1]^d)\le r_k^{-(s+\eta)}$. 

Also, by hypothesis, there are $\hat y\in\hat F$ and $\theta>0$ such that $B(\hat y, 2\theta)\in [0,1]^d$ and thus, for $k$ sufficiently large there is a point $y_k\in T_k(F)$ so that $B(y_k,\theta)\subset[0,1]^d$. It follows that
$$N_{3r_k}(T_k(F)\cap B(y_k,\theta))\le r_k^{-(s+\eta)}.$$
Defining $r=3r_ka_k^{-1}$, $R=\theta b_k^{-1}$ and $z_k=T_k^{-1}(y_k) \in F$, we get $$N_r(F\cap B(z_k,R))\le \left(\frac{R}{r}\right)^{s+\eta}$$
 (and hence $\dim_L F\le s$). It is easily seen that $r\le R^{1+\epsilon/2}$, therefore, $\dim_{qL} F\le s$.

Finally, suppose that $\dim _{qL}\hat{F}=t$, so given $\eta>0$ there is some $0<\delta<1$ and arbitrarily small $r, R$ with $r\le R^{1+\delta}$ and $\hat y\in \hat F$ such that$$N_r(B(\hat y, R)\cap \hat F)\le (R/r)^{t+\eta}.$$ The geometric condition involving $\theta$ allows us to assume that $B(\hat y, R)\subset [0,1]^d$. Now choose $k$ such that $Cb_{k+1}^{-\epsilon}\le r\le Cb_k^{-\epsilon}$, where $C,\epsilon$ are as before. Then, there is $y\in T_k(F)\cap [0,1]^d$ such that $\|y-\hat y\|\le Cb_k^{-\epsilon}$ and moreover, $B(y, \frac{1}{2}R)\subset[0,1]^d$ (for $k$ sufficiently large). Defining $r_k=3Cb_k^{-\epsilon}$ and $R_k=R/2$, it follows that for $k$ sufficiently large, 
$$N_{r_k}\bigl(T_k(F)\cap B(y, R_k)\bigr)\le N_r(\hat F\cap B(\hat y, R))\le (R/r)^{t+\eta}\approx \left(\frac{R_k}{r_k}\right)^{t+\eta},$$ 
where in the last equivalence we used the fact that $b_{k+1}\leq C'b_{k}$ for some constant $C'$. For an appropiate $z_k\in F$ and a constant $C''$, we get
$$N_{\frac{r_k}{a_k}}\bigl(F\cap B(z, \frac{R_k}{b_k})\bigr)\le C'' \left(\frac{R_k/b_k}{r_k/a_k}\right)^{t+\eta}.$$
It is easily seen that $r_k/a_k\le (R_k/b_k)^{1+\delta/2}$ for $k$ sufficiently large, and therefore $\dim_{qL} F\le \dim_{qL}\hat F$. 
\end{proof}

\begin{corollary}\label{corollary}
Suppose that $\hat F\subset \rr^d$ is a generalized fast tangent of $F$ such that $\dim _B\hat{F}=s$. Then $\dim _{qL}F\leq s\leq \dim _{qA}F$ whenever $\hat F$ contains an interior point of $[0,1]^d$. 
%
%
%
%
%
%
\end{corollary}

\begin{remark}\label{remtheotan} 
The statements from Theorem \ref{theotan} can be improved.
\begin{enumerate}[a)]
\item  Part (i) only needs the one-sided hypothesis, $p_{H}(\hat{F}, %
 T_{k}(F))\leq Cb_{k}^{-\epsilon } \rightarrow 0$, i.e., $\hat F$ is a fast pseudo-tangent of $F$. This is immediate from the proof. 
 
\item A quick inspection of the proof of (ii) shows that we have $\dim_L F\le\dbs \hat{F}$ even if the convergence to the generalized tangent is not fast.
\end{enumerate}

\end{remark}

\begin{remark}
 Consider the following simple example.
Suppose that $F\subset[0,1]$ and that $(a,b)\subset[0,1]$ with $a,b\in F$, but $(a,b)\cap F=\emptyset$. By considering $T_kx=2^k(x-a)$, it is easily seen that $\hat{F}=\{0\}$ is a fast tangent of $F$, so, unless $\dim_{qL}F=0$, the conclusion in (ii) is false in this case. 
This example illustrates that for the quasi-lower Assouad dimension, an additional hypothesis that ensures the tangent `carries' information about the interior of $F$ is necessary. In the statement of (ii), we have chosen to put this hypothesis directly on $\hat F$. Alternatively, we could have put an additional hypothesis on the approximations of the tangent, for example requiring that $p_H(F\cap B(z_{k},Ca_{k}^{-1}),\hat{F})\le Cb_k^{-\epsilon}$, where $z_k\in F$. The proof is a slight modification of the one given here.
\end{remark}

As an application of our results, we calculate the quasi-Assouad dimensions of a class of planar self-affine sets. Recall that an iterated function system (IFS) is a family $\{f_{1},\ldots ,f_{m}\}$ of contractions $f_{i}:\rr^d\rightarrow \rr^d$, and that the attractor of the IFS  is the unique
non-empty compact set $E$ that satisfies the identity 
\begin{equation*}
E=\bigcup_{i=1}^{m}f_{i}(E).
\end{equation*}
If the maps of the IFS are contracting similarities (affine maps), the attractor is called self-similar set (self-affine set, respectively).

In \cite[Sec. 2.3]{Fr}, Fraser determines the (lower) Assouad dimensions of
self-affine carpets that are the attractor of an IFS in the extended Lalley-Gatzouras and Bara\'{n}ski classes. These fractals, denoted $E,$ are
generated by an IFS with maps of the form $S_{i}(x,y)=(c_{i}x,$ $d_{i}y)+(a_{i},b_{i})$, for some $c_{i},d_{i}\in (0,1)$,  $1\le i\le m$, where $c_i\neq d_i$ for at least one $i$. (See \cite{Fr} for their complete definitions.)  Let $\pi{_1}$ denote the projection onto
the $x$-axis and $\pi_{2}$ the projection onto the $y$-axis. Let $%
{\rm Slice}_{1,i}(E)$ (resp., ${\rm Slice}_{2,i}(E)$) be the vertical (horizontal) slice of $E$ through the fixed point of $%
S_{i}$. Fraser proves that if the self-affine carpet $E$ is of mixed
type, i.e.,  there are $i\neq i'$ such that $c_{i}> d_{i}$ and $c_{i'}< d_{i'}$, then 
\begin{eqnarray}
\dim _{A}E &=&\max_i\max_{k=1,2}\Bigl(\dim _{B}\pi _{k}(E)+\dim _{B}{\rm Slice}_{k,i}(E)\Bigr),\notag \\
\dim _{L}E &=&\min_i\min_{k=1,2}\Bigl(\dim _{B}\pi _{k}(E)+\dim _{B}{\rm Slice}_{k,i}(E)\Bigr).\notag
\end{eqnarray}

Applying Corollary \ref{corollary} we obtain the following result.

\begin{proposition} For the above carpets we have
$\dim _{qA}E=\dim _{A}E$,
and similarly for the quasi-lower Assouad dimension.
\end{proposition}
\begin{proof}
To see this, we give the following sketch of the proof (see \cite[Sec 7.2]{Fr} for more details on the definitions). We assume $\max_i\max_{k=1,2}\Bigl(\dim _{B}\pi _{k}(E)+\dim _{B}{\rm Slice}_{k,i}(E)\Bigr)=\dim _{B}\pi _{1}(E)+\dim _{B}{\rm Slice}_{1,i}(E)$ for some $1\le i\le m$, that now we fix. Consider the approximate square $%
Q_{k}(i,j)$, centred at the point $\bigcap_{l\ge1}S_{j}^{k}\circ
S_{i}^{l}([0,1]^{2})$, with height $d_{j}^{k}$ and width $%
c_{j}^{k}c_{i}^{l(k)}$, where $l(k)$ is an integer chosen so that \[
c_{j}^{k}c_{i}^{l(k)+1}\leq d_{j}^{k}\leq c_{j}^{k}c_{i}^{l(k)}.\] Take
the maps $T_{k}$ that stretch by $d_{j}^{-k}$ in height and by $\left(
c_{j}^{k}c_{i}^{l(k)}\right) ^{-1}$ in width, and map the corner of $Q_{k}$
to the origin. One can check these maps satisfy the required Lipschitz
properties with $b_{k}=d_{j}^{-k}$ and $a_{k}=\left(
c_{j}^{k}c_{i}^{l(k)}\right) ^{-1}$. Take $F_{i}=\pi _{1}(E)\times \pi
_{2}({\rm Slice}_{1,i}(E))$. This is a product of two self-similar sets satisfying
the open set condition and hence \[\dim _{B}F_{i}=\dim
_{B}\pi _{1}(E)+\dim _{B}{\rm Slice}_{1,i}(E).\] From the structure of the carpet, and since $l(k)>k\log_{c_i}(d_j/c_j)$,
it can be seen that \[{\rm dist}_{H}(T_{k}(Q_{k}),F_i)\leq
\max_{n}d_{n}^{l(k)}\le d_{j}^{k \beta}\] for a suitable $\beta >0$. 
Appealing to the corollary gives the result.
\end{proof}


We finish the section with the following remark on the quasi-Assouad dimension of self-similar sets.

\begin{remark}\label{remark:self-similar}
The weak separation property (WSP) is a separation property (on an IFS) that, although weaker than the classical open set condition, ensures nice properties on the attractor; see cite \cite{Ze} for the definition. This property has been essential to describe the behaviour of the Assouad dimension of self-similar sets. If it holds, then $\dim _{B}E=\dim_{A}E $, so in particular, $\dim_{qA} E=\dim_B E$. Moreover, if it does not hold, then $\dim_A E\ge1$; see \cite{FHOR}.  These results establish, in $\rr$, the precise dichotomy that $\dim_A E$ is either $\dim_BE$ or $1$ depending on whether the WSP holds or not.

In absence of the WSP, the behaviour of the quasi-Assouad dimension is unknown. However, it still remains valid in $\rr$ that $\dim_{qA} E=\dim_B E$ in the case that the IFS {\em does not have super-exponential concentration of cylinders}. This is because, under this hypothesis, equality $\dim_A^\theta E=\dim_B E$ holds for all $0<\theta<1$ by \cite[Corollary 4.2]{FY1}, and also $\dim_{qA} F=\lim_{\theta\to1}\dim_A^\theta F$ for any  $F\subset\rr$ by \cite[Corollary 2.2]{FTale}.  

The super-exponential concentration of cylinders property was introduced by Hochman in his celebrated paper \cite{H} to give a substantial improvement on the folklore conjecture that, in absence of exact overlaps, the Hausdorff (and box) dimension of a self-similar set coincides with its similarity dimension. 
This property is verified (trivially) in the case that the IFS produces exact overlaps. Although it seems difficult to check, in general, there are interesting overlapping examples that do not have super-exponential concentration. In particular, there are examples of self-similar sets with box dimension smaller than 1, that do not verify the WSP but also do not have super-exponential concentration of cylinders. This implies that for the quasi-Assouad dimension there is no such dichotomy as that mentioned above for the Assouad dimension case.
\end{remark}

\section{Further properties}\label{section:further_results}

\subsection{Dimensions of sequences with decreasing gaps}

In \cite{GHM} it was shown that sequences in $\mathbb{R}$ with decreasing
gaps have Assouad dimension $0$ or $1$. It is easy to see that the same
statement is true for the quasi-Assouad dimension.

\begin{proposition}
(i) If $\overline{\dim}_B\: E=0$, then $\dim _{qA}E=0$.

(ii) If $E=\{a_{j}\}_{j}$ $\subseteq \lbrack 0,1]$, where $\{
a_{j}-a_{j+1}\}_{j}$ is a decreasing sequence, then $\dim _{qA}E=1$ if $%
\overline{\dim}_B\: E>0$ and otherwise $\dim _{qA}E=0$.
\end{proposition}

\begin{proof}
(i) Fix $\varepsilon ,\delta >0$. The assumption that $\overline{\dim}_B\:
E=0$ ensures that for all sufficiently small $r$, $N_{r}(E)\leq r^{-\delta
\varepsilon }$. Thus for any $R\geq r^{1-\delta }$ and any $x\in E$, we have 
\begin{equation*}
N_{r}(B(x,R)\cap E)\leq N_{r}(E)\leq r^{-\delta \varepsilon }\leq \left( 
\frac{R}{r}\right) ^{\varepsilon },
\end{equation*}%
from whence the conclusion is immediate.

(ii) In \cite[Thm. 6.2]{FY1} it is shown that if $E$ is a sequence with
decreasing gaps, then for all $\theta \in (0,1)$, $\dim _{A}^{\theta }E=\min
\left( \frac{\dbs E}{1-\theta },1\right) $. Thus if $%
\dbs E>0$, then $\dim _{qA}E\geq \sup_{\theta <1}\dim
_{A}^{\theta }E=1$.
\end{proof}

\begin{remark}
Although it is also true that $\dim _{A}E=0$ or $1$ for sequences with
decreasing gaps, the criterion is different. Indeed, as noted in \cite[Ex.
6.3]{FY2}, $E=\{e^{-\sqrt{n}}\mathbb{\}}_{n}$ is a set with decreasing gaps
having $\dbs E=0=\dim _{qA}E$, but $\dim _{A}E=1$.
\end{remark}

\subsection{Quasi-lower Assouad dimension and Hausdorff dimension}

The following proposition establishes the relation between these dimensions. 

\begin{proposition}
\label{Propo-qLHau} If $E$ is a closed subset of $\mathbb{R}^d$, then $\dim _{qL}E\leq \dim _{H}E$.
\end{proposition}

\begin{proof}
Our proof is based on the method of proof of \cite[Theorem 6]{L2}. If $\dim_{qL} E=0$ there is nothing to prove, so assume $\alpha<\dim_{qL}E$ for some $\alpha>0$ and pick any small $\delta>0$ such that $\alpha<\underline{h_E}(\delta)$. We will show $\dim_H E\ge\alpha/(1+\delta)$. This will prove the proposition.
%

Recall that the $r$-packing number of a subset $F\subset\mathbb{R}^d$, $P_{r}(F)$, is the maximum number of disjoint balls of radius $r$ centred in $F$. It is easily seen that there is a  constant $c>0$ such that for any bounded $F\subseteq \rr^d$ and $r>0$, we have $P_{2r}(F)\ge cN_r(F)$. Therefore, since $\alpha<\underline{h_E}(\delta)$, there is $\rho_\delta>0$ such that for any $x\in E$ and any $%
r\leq R^{1+\delta }\leq R\leq \rho _{\delta },$ 
\begin{equation*}
P_{2r}(B(x,R)\bigcap E)\geq (R/r)^{\alpha }.
\end{equation*}%
In particular, $P_{2R^{1+\delta }}(B(x,R)\bigcap E)\geq R^{-\delta \alpha }$.

Fix $x\in E$ and $R_{1}\leq \rho _{\delta }$. There are $x_{1},\ldots
,x_{R_{1}^{-\delta \alpha }}$ points in $E\bigcap B(x,R_{1})$ such that the
balls $B(x_{j},2R_{1}^{1+\delta })$ are disjoint for $j=1,\ldots
,R_{1}^{\delta \alpha }$. Now let $R_{2}=R_{1}^{1+\delta }$ and notice that 
\begin{equation*}
B(x_{j},2R_{2})\subseteq B(x,R_{1}+2R_{1}^{1+\delta })\subseteq B(x,2R_{1})
\end{equation*}%
(as we can take $2R_{1}^{\delta }<1$). We let $C_{0}=B(x,2R_{1})$, $%
C_{1}=\bigcup_{j=1}^{R_{1}^{-\delta \alpha }}B(x_{j},2R_{2})$ and refer to
the balls $B(x_{j},2R_{2})$ as the Cantor balls of level $1$.

Repeating this procedure, we see that for each $j$, 
\begin{equation*}
P_{2R_{2}^{1+\delta }}(B(x_{j},R_{2})\bigcap E)\geq R_{2}^{-\delta \alpha }
\end{equation*}%
so there are $x_{j,1},\ldots ,x_{j,R_{2}^{-\delta \alpha }}\in
B(x_{j},R_{2})\bigcap E$, such that $\{B(x_{j,k,}2R_{2}^{1+\delta
})\}_{k=1}^{R_{2}^{-\delta \alpha }}$ are pairwise disjoint. Put $%
R_{3}=2R_{2}^{1+\delta }$. Furthermore, 
\begin{equation*}
B(x_{j,k},2R_{3})\subseteq B(x_{j},R_{2}+2R_{3})\subseteq
B(x_{j},2R_{2})\subseteq C_{1},
\end{equation*}%
so all these balls are disjoint. Let 
\begin{equation*}
C_{2}=\bigcup_{k=1}^{R_{2}^{-\delta \alpha }}\bigcup_{j=1}^{R_{1}^{-\delta
\alpha }}B(x_{j,k},2R_{3})
\end{equation*}%
and call these the Cantor balls of level 2.

Inductively, given disjoint balls $B(x_{j_{1},\ldots ,j_{k-1}},2R_{k})$, then, for $l=1,\ldots ,R_{k}^{-\delta \alpha },$ we find points 
\begin{equation*}
x_{j_{1},\ldots ,j_{k-1},l}\in B(x_{j_{1},\ldots
,j_{k-1}},R_{k})\bigcap E
\end{equation*} 
such that $\{B(x_{j_{1},\ldots ,j_{k-1},l},2R_{k}^{1+\delta })\}_{l}$ are
disjoint. Put $R_{k+1}=R_{k}^{1+\delta }$ and let 
\begin{equation*}
C_{k}=\bigcup_{\substack{ j_{i}\in \{1,\ldots ,R_{i}^{-\delta \alpha }\}  \\ %
i=1,\ldots ,k}}B(x_{j_{1},\ldots ,j_{k}},2R_{k+1})\subseteq C_{k-1},
\end{equation*}%
$C_{k}$ being a union of balls of level $k$.

We have $R_{k+1}=R_{k}^{1+\delta }=R_{1}^{(1+\delta )^{k}}$. Also, notice $%
C_{k}$ is disjoint union of $M$ balls, where 
\begin{equation*}
M=\prod_{j=1}^{k}R_{j}^{-\delta \alpha }=\prod_{j=1}^{k}R_{1}^{-\delta
\alpha (1+\delta )^{j-1}}=R_{1}^{-\alpha ((1+\delta )^{k}-1)}.
\end{equation*}%
Let $C=\bigcap_{k=1}^{\infty }C_{k}$. As each element of $C$ is a limit
point of the centre of the Cantor balls and $E$ is closed, then $C\subseteq
E $. We will use the mass distribution principle to check $\dim _{H}C\geq
\alpha /(1+\delta )$; see \cite[Proposition 2.1]{Fal}.

Let $\mu $ be the probability measure that assigns equal mass on the Cantor
balls of each level, i.e., each ball in $C_{k}$ gets measure $%
M^{-1}=R_{1}^{\alpha ((1+\delta )^{k}-1)}$. We want to show that there is
some constant $A=A(\alpha ,E)$ such that $\mu (U)\leq A(\text{diam}(U))^{%
\frac{\alpha }{1+\delta }}$ for all Borel sets $U.$

Without lost of generality we assume $U=B(y,r)$, where $2R_{k+1}<r\leq
2R_{k} $, $y\in E$. Any ball of radius $2R_{k}$ that intersects $U$ will
have its centre in $B(y,4R_{k})$. Since $X$ is doubling, there is a constant 
$A_{1}$ such that $P_{2R_{k}}(B(y,4R_{k}))\leq A_{1}$ for all $y\in E$ and
all $k$. As Cantor balls at level $k-1$ are disjoint, of radius $2R_{k}$ and
centred in $E$, at most $A_{1}$ of such balls can intersect $U$. Thus $U$
intersects at most $A_{1}R_{k}^{-\delta \alpha }$ level $k$ Cantor balls, so 
\begin{equation*}
\mu (U)\leq A_{1}R_{k}^{-\delta \alpha }\cdot R_{1}^{\alpha ((1+\delta
)^{k}-1)}=\frac{A_{1}}{R_{1}^{\alpha }}R_{1}^{\alpha (1+\delta )^{k-1}}\leq
A(\text{diam}(U)^{\frac{\alpha }{1+\delta }}.
\end{equation*}%
Hence the mass distribution principle implies $\alpha
/(1+\delta )\leq \dim _{H}C\leq \dim _{H}E,$ completing the proof.
\end{proof}

\subsection{Different values for different dimensions}

Provided the upper and lower box dimensions are to be distinct,
then given any six numbers in $[0,1],$ appropriately ordered, there is a
compact set $E\subseteq \lbrack 0,1]$ which have those numbers as the
Assouad-type and box dimensions. We will construct Cantor sets to illustrate
this, making use of the following formulas for the box and quasi-Assouad
dimensions of a Cantor set $E$ with ratios of dissection $r_{k}$ at step $k$%
; see \cite{GHM} and \cite{LX}. 
\begin{equation*}
\dbs E=\limsup_{n}\frac{n\log 2}{|\log r_{1}\cdot \cdot
\cdot r_{n}|},\text{ }\dim _{A}E=\limsup_{n}\sup_{k}\frac{n\log 2}{|\log
r_{k+1}\cdot \cdot \cdot r_{k+n}|}
\end{equation*}
If $\inf r_k>0$ , then 
\begin{equation*}
\dim _{qA}E=\lim_{\delta \rightarrow 0}\limsup_{n}\sup_{k\in S_{n,\delta }}%
\frac{n\log 2}{|\log r_{k+1}\cdot \cdot \cdot r_{k+n}|},
\end{equation*}%
where $S_{n,\delta }=\{k:r_{k+1}\cdot \cdot \cdot r_{k+n}\leq (r_{1}\cdot
\cdot \cdot r_{k})^{\delta }\}$. For the lower box and (quasi)-lower Assouad
dimensions replace sup and lim sup by inf and lim inf respectively.

\begin{example}
\label{strict}Assume $1\leq a\leq \alpha \leq u<v\leq \beta $ $\leq b<\infty 
$ are given. There is a Cantor set $E\subseteq \lbrack 0,1]$ with $\dim
_{A}E=1/a,$ $\dim _{qA}E=1/\alpha $, $\dbs E=1/u$, $%
\dbi E=1/v$, $\dim _{qL}E=1/\beta $ and $\dim _{L}E=1/b$.

The example we will construct is a generalization of \cite[Ex. 1.18]{LX} and
so we will only sketch the ideas. To begin, we choose a sequence of integers 
$s_{j}$ tending to infinity very rapidly. For convenience, put $%
t_{2j}=s_{2j}\left( \frac{v-\alpha}{u-\alpha }\right) $ and $t_{2j+1}=s_{2j+1}%
\left( \frac{\beta-u}{\beta -v}\right) $. (If $u=\alpha$ put $t_{2j}=s_{2j}$ and similarly if $v=\beta$.) We will define the ratios of
dissection at the various steps as follows:%
\begin{equation*}
\begin{array}{cc}
\text{Ratio} & \text{At steps} \\ 
2^{-v} & t_{2j-1}+j,\dots ,s_{2j} \\ 
2^{-\alpha } & s_{2j}+1,\ldots ,t_{2j} \\ 
2^{-a} & t_{2j}+1,\ldots ,t_{2j}+j \\ 
2^{-u} & t_{2j}+j+1,\ldots ,s_{2j+1} \\ 
2^{-\beta } & s_{2j+1}+1,\ldots ,t_{2j+1} \\ 
2^{-b} & t_{2j+1}+1,\ldots ,t_{2j+1}+j%
\end{array}%
\end{equation*}%
Provided $s_{j}$ tends to infinity sufficiently quickly, the ratios at steps 
$\{t_{2j}+1,\ldots ,t_{2j}+j\}$ and $\{t_{2j+1}+1,\ldots ,t_{2j+1}+j\}$ will
not influence the long run averages that determine the box and quasi-(lower)
Assouad dimensions. But these ratios will determine the (lower) Assouad
dimensions. The construction ensures that the quasi-Assouad dimension is
determined by choosing $R$ to be the length of the Cantor intervals at step $%
s_{2j}$ and $r$ to be the length of Cantor intervals at step $t_{2j}$, while
the quasi-lower dimensions arise with $R$ the length at step $s_{2j+1}$ and $%
r$ the length at step $t_{2j+1}$. The choice of $t_{2j}$ and $t_{2j+1}$ are
made to ensure that the geometric means of the ratios stay within the range 
$[2^{-v},2^{-u}]$ (in the limit) so that the box dimensions are determined
along the subsequences of lengths of Cantor intervals at steps $s_{j}$. The
details are left to the reader.
\end{example}

\subsection{Dimensions of orthogonal projections}

For our last example, we will show that, as in the case of Assouad dimension
(see \cite{Fr} and also \cite{FO}), the quasi-Assouad dimension may increase
under orthogonal projections. As before, we will let $\pi _{x}$ (resp., $\pi
_{y})$ denote the projection onto the $x$ (resp. $y)$ axis.

\begin{proposition}
There is a subset $E\subseteq \mathbb{R}^{2}$ such that \[\dim _{qA}\pi
_{x}(E)=1>1/2=\dim _{qA}E.\]
\end{proposition}

\begin{proof}
We will construct an example to show this. For each $j$ and $i=1,\ldots
,2^{j},$ let $x_{ij}=2^{-j}+(i-1)2^{-2j}$. The points $x_{ij}$ belong to $%
[2^{-j},2^{-(j-1)})$ and are spaced $2^{-2j}$ apart. Let $y_{ij},i=1,\ldots
,2^{j},$ be the endpoints, ordered from left to right, of the gaps created
at step $j$ in the standard construction of the Cantor set with ratio of
dissection $1/4$; these gaps have length $2^{-2j+1}$. Put $%
E_{j}=\{(x_{ij},y_{ij}):i=1,\ldots ,2^{j}\}$ and $E=\cup _{j}E_{j}$. See
Figure 2 below. Of course, $\pi _{x}(E)=\bigcup
_{j}\{x_{ij}:i=1,\ldots ,2^{j}\}$. By checking $N_{r}(B(x_{0},R)\cap \pi
_{x}(E))$ for $x_{0}=2^{-j},$ $R=2^{-j}$ and $r=2^{-2j},$ it is easy to see
that $\dim _{qA}\pi _{x}(E)=1$. 

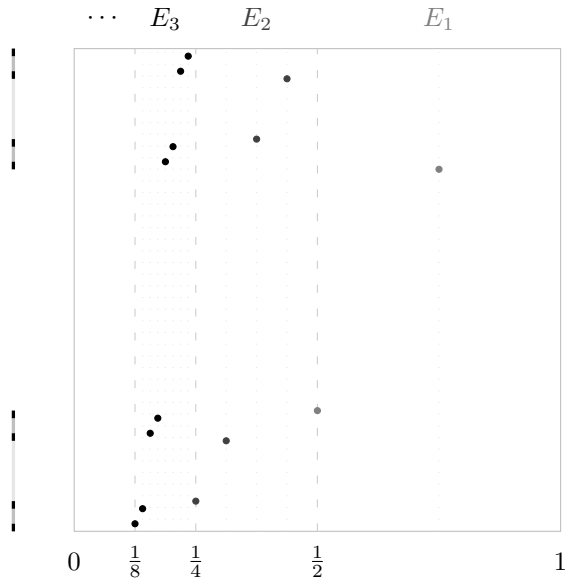
\begin{figure}[tbp]\label{figureb}
\centering
\begin{tikzpicture}[scale=.4]
\draw[draw=gray!50] (0,0) rectangle (16,16);

\draw[loosely dashed, draw=gray!40] (8,0) -- (8,16);
\draw[loosely dotted, draw=gray!20] (12,0) -- (12,16);
\filldraw[black!50](8,4) circle (.10cm);
\filldraw[black!50](12,12) circle (.10cm);

\draw[loosely dashed, draw=gray!40] (4,0) -- (4,16);
\draw[loosely dotted, draw=gray!20] (5,0) -- (5,16);
\draw[loosely dotted, draw=gray!20] (6,0) -- (6,16);
\draw[loosely dotted, draw=gray!20] (7,0) -- (7,16);
\filldraw[black!75](4,1) circle (.10cm);
\filldraw[black!75](5,3) circle (.10cm);
\filldraw[black!75](6,13) circle (.10cm);
\filldraw[black!75](7,15) circle (.10cm);

\draw[loosely dashed, draw=gray!40] (2,0) -- (2,16);
\draw[loosely dotted, draw=gray!20] (2.25,0) -- (2.25,16);
\draw[loosely dotted, draw=gray!20] (2.5,0) -- (2.5,16);
\draw[loosely dotted, draw=gray!20] (2.75,0) -- (2.75,16);
\draw[loosely dotted, draw=gray!20] (3,0) -- (3,16);
\draw[loosely dotted, draw=gray!20] (3.25,0) -- (3.25,16);
\draw[loosely dotted, draw=gray!20] (3.5,0) -- (3.5,16);
\draw[loosely dotted, draw=gray!20] (3.75,0) -- (3.75,16);
\filldraw(2,.25) circle (.10cm);
\filldraw(2.25,.75) circle (.10cm);
\filldraw(2.5,3.25) circle (.10cm);
\filldraw(2.75,3.75) circle (.10cm);
\filldraw(3,12.25) circle (.10cm);
\filldraw(3.25,12.75) circle (.10cm);
\filldraw(3.5,15.25) circle (.10cm);
\filldraw(3.75,15.75) circle (.10cm);

\draw[black!50](12,17) node {$E_1$};
\draw[black!75](6,17) node {$E_2$};
\draw(3,17) node {$E_3$};
\draw(1,17) node {$\cdots$};

\draw[very thick, gray!20] (-2,0) -- (-2,4);
\draw[very thick, gray!20] (-2,12) -- (-2,16);

\draw[very thick, gray!50] (-2,0) -- (-2,1);
\draw[very thick, gray!50] (-2,3+0) -- (-2,3+1);
\draw[very thick, gray!50] (-2,12+0) -- (-2,12+1);
\draw[very thick, gray!50] (-2,15+0) -- (-2,15+1);

\draw[very thick] (-2,0) -- (-2,.25);
\draw[very thick] (-2,.75) -- (-2,1);
\draw[very thick] (-2,3+0) -- (-2,3+.25);
\draw[very thick] (-2,3+.75) -- (-2,3+1);
\draw[very thick] (-2,12+0) -- (-2,12+.25);
\draw[very thick] (-2,12+.75) -- (-2,12+1);
\draw[very thick] (-2,12+3+0) -- (-2,12+3+.25);
\draw[very thick] (-2,12+3+.75) -- (-2,12+3+1);

\draw(0,-1) node {$0$};
\draw(16,-1) node {$1$};
\draw(8,-1) node {$\frac{1}{2}$};
\draw(4,-1) node {$\frac{1}{4}$};
\draw(2,-1) node {$\frac{1}{8}$};

\end{tikzpicture}
\caption{The sets $E_{1}$, $E_{2}$ and $E_{3}$ in the construction of $E$.}
\end{figure}

To determine the quasi-Assouad dimension of $E$, it is convenient to take as
the definition of a `ball', $B(x_{0},R)$ in $\mathbb{R}^{2},$ the square
with centre $x_{0}$ and sides of length $R$. Fix such a ball with $%
x_{0}=(x_{ij},y_{ij})\in E$ and assume $2^{-2(s+1)}<R\leq 2^{-2s}$ for some $%
s\in \mathbb{N}$. The size of $R$ ensures that the interval $\pi
_{y}(B(x_{0},R))$ can intersect only one Cantor interval of step $s$. Choose
any $r<R$, say $2^{-2t}<r\leq 2^{-2(t-1)}$.

First, note that%
\begin{equation*}
\bigl(B(x_{0},R)\cap E\bigr)\cap \bigl(\lbrack 0,2^{-2t})\times \lbrack 0,1]%
\bigr):=\Omega
\end{equation*}%
is contained in $[0,2^{-2t})\times ($union of Cantor intervals of step $t$
contained in $\pi _{y}(B(x_{0},R))$. There are at most $2^{t-s}$ such Cantor
intervals, each of length $2^{-2t}$. Hence $N_{r}(\Omega )\leq 2^{t-s}$.

Next, for each $m\in \{t+1, \ldots,2t\}$ consider the elements of 
\begin{equation*}
\bigl( B(x_{0},R)\cap E\big) \cap \bigl(\lbrack 2^{-m},2^{-(m-1)})\times
\lbrack 0,1]\bigr):=\Omega _{m}.
\end{equation*}%
The $y$-coordinates of these points are the endpoints of the gaps at step $m$
lying within the one Cantor interval of step $s$ that $\pi _{y}(B(x_{0},R))$
intersects. There are $2^{m-t}$ of these contained within each Cantor
sub-interval of step $t$. As the distance between consecutive $x$%
-coordinates is $2^{-2m}$, the total horizontal distance between the points
whose $y$-coordinates lie in a (fixed) Cantor interval of step $t$ is $%
2^{-2m}(2^{m-t})=2^{-(m+t)}\leq 2^{-2t}$, while the total vertical distance
is the width of the Cantor subinterval, $2^{-2t}$. Consequently, such points
lie within a square of side length $2^{-2t}$ and hence we can cover $\Omega
_{m}$ with $2^{t-s}$ squares of side length $2^{-2t}$ for each such $m$.

Finally, observe that the cardinality of the remainder of $B(x_{0},R)\cap E,$
which is contained in $[2^{-t},1]\times \lbrack 0,1]$, is dominated by $2$
times the number of gaps of step $\leq t$ within a Cantor interval of step $%
s,$ and this is bounded above by $2^{t-s+1}$. Combining together all these
observations, we see that 
\begin{equation*}
N_{r}(B(x_{0},R)\cap E)\leq (t+1)2^{t-s}+2^{t-s+1}.
\end{equation*}%
It follows that for each $\delta >0$, $\overline{h_{E}}(\delta )\leq 1/2$
and thus $\dim _{qA}E\leq 1/2$. It is not difficult to see that these
estimates are essentially sharp and thus we actually have $\dim _{qA}E=1/2$. 
\end{proof}

\begin{remark}
The classical Marstrand projection theorem, and its more recent variants, (c.f. \cite{Fal15}) states that the orthogonal
projections of planar sets have the same dimension at almost every angle,
where here dimension can be Hausdorff, upper/lower box or packing. This is
not the case for the Assouad dimension. Indeed, it is shown in \cite{FO} that for any $s$ with $\log_53<s<1$, there exists a self-similar set $F \subset \mathbb{R}^2$ and two non-empty intervals $I, J$ such that $\dim_A \pi_{\theta} F=s$ for all $\theta \in I$, while $\dim_A \pi_{\theta}F=1 $ for almost all $\theta \in J$. Here $\pi_{\theta}$ denotes the projection onto the line passing through the origin with angle $\theta$.

It is unknown what the situation is for the quasi-Assouad dimension. The set $F$ from \cite{FO} is not helpful in resolving this problem. Since it is a Sierpinski triangle with contraction factor $c$, for some $c\in(1/5,1/3)$, its orthogonal projections are self-similar sets $F_t$ attractors of iterated functions systems of the form $\{cx, cx+1-c, cx+t\}$, up to rescaling. Then, if $c$ is algebraic, for almost every $t$, these projections does not have super exponential contraction of cylinders (see Theorem 1.6 in \cite{H} and the comment after its proof),  and hence by Remark\ref{remark:self-similar}, for almost every $t$  we have $\dim_{qA} F_t=\dim_B F_t$, and hence the Marstrand projection theorem imply that the quasi-Assouad dimensions of the projections of $F$ are constant almost everywhere. On the other hand, if $c$ is not algebraic, it is unknown if there are super exponential contraction of cylinders in the projections, so it is unknown their quasi-Assouad dimension.
\end{remark}



\begin{section}{Acknowledgements}
The work of I. Garc\'ia was partially supported by a grant from the Simons Foundation. The work of K. Hare was supported by NSERC 2016 03719. 
\end{section}

\end{document}